\batchmode
\makeatletter
\def\input@path{{"/Users/russw/Documents/Research/mypapers/Strict Erdos-Ko-Rado theorems for simplicial complexes/"}}
\makeatother
\documentclass[12pt,oneside,english,lowtilde]{amsart}
\usepackage[T1]{fontenc}
\usepackage[latin9]{inputenc}
\usepackage{babel}
\usepackage{url}
\usepackage{enumitem}
\usepackage{amstext}
\usepackage{amsthm}
\usepackage{amssymb}
\usepackage{geometry}
\usepackage[dvips,pdfusetitle,
 bookmarks=true,bookmarksnumbered=false,bookmarksopen=false,
 breaklinks=false,pdfborder={0 0 1},backref=false,colorlinks=false]
 {hyperref}
\usepackage{breakurl}

\makeatletter
\numberwithin{equation}{section}
\numberwithin{figure}{section}
\theoremstyle{plain}
\newtheorem{thm}{\protect\theoremname}[section]
\theoremstyle{plain}
\newtheorem{conjecture}[thm]{\protect\conjecturename}
\theoremstyle{plain}
\newtheorem{cor}[thm]{\protect\corollaryname}
\theoremstyle{plain}
\newtheorem{lem}[thm]{\protect\lemmaname}
\theoremstyle{definition}
\newtheorem{example}[thm]{\protect\examplename}
\theoremstyle{plain}
\newtheorem{prop}[thm]{\protect\propositionname}
\theoremstyle{remark}
\newtheorem{rem}[thm]{\protect\remarkname}

\usepackage{lmodern}

\usepackage{color}

\makeatother

\providecommand{\conjecturename}{Conjecture}
\providecommand{\corollaryname}{Corollary}
\providecommand{\examplename}{Example}
\providecommand{\lemmaname}{Lemma}
\providecommand{\propositionname}{Proposition}
\providecommand{\remarkname}{Remark}
\providecommand{\theoremname}{Theorem}

\begin{document}
\global\long\def\link{\operatorname{link}}%

\global\long\def\del{\operatorname{del}}%

\global\long\def\cone{\operatorname{Cone}}%

\global\long\def\depth{\operatorname{depth}}%

\global\long\def\shift{\operatorname{Shift}}%

\global\long\def\bdry{\partial}%

\global\long\def\symdiff{\ominus}%

\global\long\def\shiftext{\shift^{\wedge}}%

\global\long\def\extalg{\bigwedge}%

\global\long\def\ff{\mathbb{F}}%

\title{Strict Erd\H{o}s-Ko-Rado theorems for simplicial complexes}
\author{Denys Bulavka and Russ Woodroofe}
\thanks{Work of the first author is supported in part by the Israel Science
Foundation grant ISF-2480/20. Work of the second author is supported
in part by the Slovenian Research Agency (research program P1-0285
and research projects J1-9108, N1-0160, J1-2451, J1-3003, and J1-50000).}
\address{Einstein Institute of Mathematics, Hebrew University, Jerusalem 91904,
Israel}
\email{Denys.Bulavka@mail.huji.ac.il}
\urladdr{\url{https://kam.mff.cuni.cz/~dbulavka/}}
\address{Univerza na Primorskem, Glagoljaška 8, 6000 Koper, Slovenia}
\email{russ.woodroofe@famnit.upr.si}
\urladdr{\url{https://osebje.famnit.upr.si/~russ.woodroofe/}}
\begin{abstract}
We show that if a simplicial complex $\Delta$ is a near-cone of sufficiently
high depth, then the only maximum families of small pairwise intersecting
faces are those with a common intersection. Thus, near-cones of sufficiently
high depth satisfy the strict Erd\H{o}s-Ko-Rado property conjectured
by Holroyd and Talbot and by Borg. One consequence is a strict Erd\H{o}s-Ko-Rado
theorem for independence complexes of chordal graphs with an isolated
vertex. Under stronger shiftedness conditions, we prove a sharper
stability theorem of Hilton-Milner type, as well as two cross-intersecting
theorems.
\end{abstract}

\maketitle

\section{\label{sec:Introduction}Introduction}

What is the largest cardinality of a family of pairwise-intersecting
sets? A 1961 result of Erd\H{o}s, Ko, and Rado answers this question
if the sets all have the same small number of elements (and have no
other restrictions).
\begin{thm}[Erd\H{o}s, Ko, and Rado \cite{Erdos/Ko/Rado:1961}]
\label{thm:EKR} Let $r\leq n/2$. If $\mathcal{A}$ is a family
of pairwise-intersecting subsets of $[n]$, each with $r$ elements,
then $\left|\mathcal{A}\right|\leq{n-1 \choose r-1}$.

If $\left|\mathcal{A}\right|$ achieves the upper bound and $r<n/2$,
then $\mathcal{A}$ consists of all the $r$ element subsets containing
some fixed element.
\end{thm}

That is, a maximum family of sufficiently small pairwise-intersecting
objects (a maximum \emph{intersecting family}) is given by a family
having a common intersection. Moreover, under a slightly stronger
hypothesis, these are the only such maximum families.

There are a large number of generalizations of Theorem~\ref{thm:EKR}.
We focus here on one broad family of such generalizations. Holroyd
and Johnson asked at the 1997 British Combinatorial Conference \cite{Holroyd:1999}
whether an analogue of Erd\H{o}s-Ko-Rado holds for independent sets
in cyclic graphs and their graph-theoretic powers. Talbot showed the
answer to be ``yes'' in the strong sense in a 2003 paper \cite{Talbot:2003}:
the largest intersecting families of $r$ element independent sets
are exactly the families of $r$ element independent sets having a
common intersection.

Holroyd and Talbot then asked whether (when $r$ is sufficiently small)
similar results hold for intersecting families of $r$ element independent
sets in other graphs \cite{Holroyd/Talbot:2005}. They found counterexamples
for $r$ around half the size of a maximum independent set, but not
for $r$ at most half of the minimal size of a maximal independent
set (the independent domination number). Borg later extended the question
of Holroyd and Talbot from independence complexes of graphs to arbitrary
simplicial complexes \cite{Borg:2009}, as follows.
\begin{conjecture}[Holroyd and Talbot \cite{Holroyd/Talbot:2005}; Borg \cite{Borg:2009}]
\label{conj:HolroydTalbotBorg} Let $\Delta$ be a simplicial complex
whose smallest facet has $d$ vertices, and let $r\leq d/2$. If $\mathcal{A}$
is a family of pairwise-intersecting faces of $\Delta$, each with
$r$ elements, then there is some vertex $v$ of $\Delta$ so that
$\left|\mathcal{A}\right|\leq f_{r-1}(\link_{\Delta}v)$.

If $\left|\mathcal{A}\right|$ achieves the upper bound and $r<d/2$,
then $\mathcal{A}$ consists of the $r$ element faces containing
some vertex $v$.
\end{conjecture}

For example, Theorem~\ref{thm:EKR} says exactly that the simplex
on $d=n$ vertices satisfies Conjecture~\ref{conj:HolroydTalbotBorg}.

If a simplicial complex $\Delta$ satisfies the upper bound of Conjecture~\ref{conj:HolroydTalbotBorg}
at a specified value of $r$, then we say that $\Delta$ is \emph{$r$-EKR}.
If every maximum intersecting family has a common intersection, then
we say that $\Delta$ is \emph{strictly $r$-EKR}. Thus, Conjecture~\ref{conj:HolroydTalbotBorg}
says that if the smallest facet of $\Delta$ has $d$ vertices, then
$\Delta$ is $r$-EKR for $r\leq d/2$, and strictly $r$-EKR for
$r<d/2$. We abuse terminology to say that a graph is (strictly) $r$-EKR
if its associated simplicial complex of independent sets has the same
property.

There has been considerable work on Conjecture~\ref{conj:HolroydTalbotBorg},
as recently surveyed in \cite{Hurlbert:2025UNP}. Borg showed in \cite{Borg:2009}
that if the smallest facet of $\Delta$ has $d$ vertices, then $\Delta$
is strictly $r$-EKR for $r$ on the order of $d^{1/3}$. A recent
preprint of Kupavskii \cite{Kupavskii:2023UNP} improves Borg's bound
from $d^{1/3}$ to within a subpolynomial factor of $d/2^{19}$. There
have been several papers \cite{Borg/Feghali:2022,Estrugo/Pastine:2021,Hilton/Holroyd/Spencer:2011,Hilton/Spencer:2009}
concerning disjoint unions of copies of powers of cyclic graphs, building
on the paper of Talbot \cite{Talbot:2003}. Olarte, Santos, Spreer,
and Stump considered Conjecture~\ref{conj:HolroydTalbotBorg} (along
with related questions) on flag manifolds \cite{Olarte/Santos/Spreer/Stump:2020}.

Hurlbert and Kamat showed in \cite{Hurlbert/Kamat:2011} that any
chordal graph with an isolated vertex satisfies the upper bound of
Conjecture~\ref{conj:HolroydTalbotBorg}. The second author showed
more generally in \cite{Woodroofe:2011a} that any near-cone satisfying
a depth condition satisfies the upper bound of Conjecture~\ref{conj:HolroydTalbotBorg}.
Here, a simplicial complex is a \emph{near-cone with apex vertex $a$}
if it is \emph{shifted with respect to $a$}; that is, if whenever
$B$ is a face and $w\in B$, also $(B\setminus w)\cup a$ is a face.
The independence complex of a graph $G$ is a cone if and only if
$G$ has an isolated vertex, and hence is also a near-cone in this
situation. The depth condition is more technical, but is well understood
on many complexes. For instance, independence complexes of chordal
graphs are vertex-decomposable \cite{Dochtermann/Engstrom:2009,Morey/Villarreal:2012,Woodroofe:2009a},
and it follows that the depth agrees with the dimension of a minimum
facet. There is a huge library of examples in the literature where
the depth is known, see e.g. \cite{Bjorner:1995,Herzog/Hibi:2011}.

Neither the paper of Hurlbert and Kamat \cite{Hurlbert/Kamat:2011}
nor that of the second author \cite{Woodroofe:2011a} addressed the
strict $r$-EKR property.

In the current paper, we prove four main theorems. The first and most
substantial theorem fills in the strict EKR gap of \cite{Hurlbert/Kamat:2011}
and \cite{Woodroofe:2011a}. We show:
\begin{thm}
\label{thm:StrictEKR}Let $\Delta$ be a simplicial complex of depth
$d-1$, and let $r<d/2$. If $\Delta$ is shifted with respect to
the vertex $a$, then $\Delta$ is strictly $r$-EKR.
\end{thm}

See Section~\ref{sec:Background} for the definition of and background
on depth. The depth of $\Delta$ is at most the minimum facet dimension
of $\Delta$, and there are examples where this inequality is strict.
The depth and minimum facet dimension agree for shifted, vertex-decomposable,
shellable, or more generally sequentially Cohen-Macaulay simplicial
complexes.
\begin{cor}
Sequentially Cohen-Macaulay near-cones satisfy Conjecture~\ref{conj:HolroydTalbotBorg},
including the strict EKR statement.
\end{cor}

The other three theorems from the paper will be a constellation of
results around the theme that if a complex is shifted with respect
to an initial set of vertices and also has sufficiently high depth,
then that complex is nearly as nice as a shifted complex.

Our second main theorem will be a stability result. The maximum families
of Theorem~\ref{thm:EKR} are \emph{stable}, in the sense that any
family that does not have a common intersection is relatively small.
Indeed, Hilton and Milner \cite{Hilton/Milner:1967} improved Theorem~\ref{thm:EKR}
to show that if $\mathcal{A}$ as in the hypothesis of that theorem
does not have a common intersection, then 
\[
\left|\mathcal{A}\right|\leq{n-1 \choose r-1}-{n-r-1 \choose r-1}+1.
\]
 The bound of Hilton and Milner is achieved by the family of subsets
consisting of $A=\{2,\dots,r+1\}$ together with all $r$ element
subsets that contain $1$ and intersect $A$ (i.e., that contain $1$
and are not contained in $A^{c}$). Borg extended the stability result
of Hilton-Milner to families of $r$ element faces of a shifted simplicial
complex \cite[Theorem 1.5]{Borg:2013}.

Our stability result will require somewhat stronger conditions on
the simplicial complex than Theorem~\ref{thm:StrictEKR}, although
still significantly weaker than the shiftedness required in \cite{Borg:2013}.
\begin{thm}
\label{thm:SimplicialHM}Let $\Delta$ be a simplicial complex of
depth $d-1$, and let $2\leq r\leq d/2$. Let $\mathcal{A}$ be an
intersecting family consisting of $r$ element faces of $\Delta$.

If $\Delta$ is shifted with respect to $\{v_{1}<\cdots<v_{r+1}\}$
and $\mathcal{A}$ has no common intersection, then 
\[
\left|\mathcal{A}\right|\leq f_{r-1}(\link_{\Delta}v_{1})-\beta+1,
\]
where $\beta$ is the number of $r$ element faces containing $v_{1}$
but no element of $\{v_{2},\dots,v_{r+1}\}$.
\end{thm}

For example, when $\Delta$ is the simplex on $n$ vertices, Theorem~\ref{thm:SimplicialHM}
recovers the bound of Hilton and Milner. The bound of Theorem~\ref{thm:SimplicialHM}
is best-possible, as it is attained by $A=\{v_{2},\dots,v_{r+1}\}$
together with all $r$ element faces that contain $v_{1}$ and intersect
$A$.

Our final main theorems concerns cross-intersecting systems of faces.
Systems $\mathcal{A},\mathcal{B}$ of sets are \emph{cross-intersecting}
if for every $A\in\mathcal{A}$ and $B\in\mathcal{B}$, it holds that
$A\cap B\neq\emptyset$. Hilton and Milner in their paper \cite[Theorem 2]{Hilton/Milner:1967}
gave an upper bound on $\left|\mathcal{A}\right|+\left|\mathcal{B}\right|$
for nonempty cross-intersecting families $\mathcal{A},\mathcal{B}$.
In particular, if $\mathcal{A}$ and $\mathcal{B}$ consist of $r$
element subsets of $\{1,\dots,n\}$, where $r\leq n/2$, then the
maximum is achieved by taking $\mathcal{B}$ to consist of a single
set, and $\mathcal{A}$ to be all sets that intersect nontrivially
with that set. Borg extended the bound of Hilton and Milner to cross-intersecting
families of faces from a shifted simplicial complex \cite[Theorem 1.2]{Borg:2010}.
We further extend Borg's result, as follows.
\begin{thm}
\label{thm:CrossIntClassic}Let $\Delta$ be a simplicial complex
of depth $d-1$, and let $2\leq r\leq d/2$. Let $\mathcal{A},\mathcal{B}$
be nonempty cross-intersecting families consisting of $r$ element
faces of $\Delta$.

If $\Delta$ is shifted with respect to $\{v_{1}<\cdots<v_{r}\}$,
then 
\[
\left|\mathcal{A}\right|+\left|\mathcal{B}\right|\leq f_{r}(\Delta)-\gamma+1,
\]
where $\gamma$ is the number of $r$ element faces containing no
element of $\{v_{1},\dots,v_{r}\}$.
\end{thm}

The message from Theorems~\ref{thm:SimplicialHM} and \ref{thm:CrossIntClassic}
is that the property of being shifted may frequently be weakened to
being shifted with respect to a small initial set of vertices, combined
with having high depth.

In a similar manner, we also extend a cross-intersecting theorem of
the authors \cite{Bulavka/Woodroofe:2024UNP} to the simplicial complex
setting. The \emph{shadow} of a system $\mathcal{B}$ of $r$ element
sets is the system $\bdry\mathcal{B}$ consisting of all $(r-1)$
element sets that are contained in at least one set of $\mathcal{B}$.
\begin{thm}
\label{thm:CrossIntShadow}Let $\Delta$ be a simplicial complex of
depth $d-1$, and let $2\leq r\leq(d+1)/2$. Let $\mathcal{A},\mathcal{B}$
be nonempty cross-intersecting families consisting respectively of
$r-1$ element faces and $r$ element faces of $\Delta$; and further
suppose that $\bdry\mathcal{B}\subseteq\mathcal{A}$.

If $\Delta$ is shifted with respect to $\{v_{1}<\cdots<v_{r}\}$,
then 
\[
\left|\mathcal{A}\right|+\left|\mathcal{B}\right|\leq f_{r-1}(\Delta)-\gamma+1,
\]
where $\gamma$ is the number of $r-1$ element faces containing no
element of $\{v_{1},\dots,v_{r}\}$.
\end{thm}

\subsection{\label{subsec:MainStrat}Main strategy}

To prove Theorems~\ref{thm:StrictEKR} and \ref{thm:SimplicialHM},
we combine combinatorial and algebraic shifting. A main difficulty
in dealing with intersection theorems in simplicial complexes via
shifting techniques is that one typically must simultaneously shift
both the simplicial complex and the intersecting family. The algebraic
shifting technique used in \cite{Fakhari:2016,Woodroofe:2011a} does
this simultaneous shifting (in one step), but is in several respects
more difficult to control precisely than the multistep combinatorial
shifting process.

Our technique here is to begin with some combinatorial shifting operations
in order to get better control over the algebraic shifting operation.
That is, we perform certain permitted combinatorial shifting operations
to simplify our intersecting family to something that is easy to control
under algebraic shifting, and then algebraically shift. The combination
will allow us to preserve the property of having no common intersection,
while requiring fewer hypotheses than combinatorial shifting by itself
would require. The crucial fact is that if a set family contains all
$r$ element subsets of an $(r+1)$ element set, then its algebraic
shift has the same property.

\subsection{Applications}

The depth condition in the hypotheses of our theorems may seem rather
arcane to the unfamiliar reader. Arcane or not, it is essential to
our arguments: depth is well-behaved under algebraic shifting, an
operation that is difficult to control directly.

Depth is well-studied in the combinatorial commutative algebra and
algebraic combinatorics literature, along with the related conditions
of (sequentially) Cohen-Macaulay and the equivalent (up to a duality)
invariants of projective dimension and Castelnuovo-Mumford regularity.
From this literature, we get a huge library of examples and applications
of our main theorems. We give a few gems from this library here, and
will go into more details about depth in Section~\ref{sec:Background},
particularly in \ref{subsec:VDcomplexes}.

Recall that a graph is \emph{chordal} if every induced cycle has length
$3$. As usual, we abuse terminology to say that a graph satisfies
an intersection condition, while meaning that its independence complex
satisfies the same condition.
\begin{cor}
Let $G$ be a chordal graph, or more generally a graph so that every
induced cycle has length $3$ or $5$. (See \cite{Woodroofe:2009a}.)
\begin{enumerate}
\item If $G$ has an isolated vertex, then it satisfies Conjecture~\ref{conj:HolroydTalbotBorg}.
\item If $G$ has at least $r$ isolated vertices, then it satisfies the
cross-intersecting conditions of Theorems~\ref{thm:CrossIntClassic}
and \ref{thm:CrossIntShadow}.
\item If $G$ has at least $r+1$ isolated vertices, then it satisfies the
stability condition of Theorem~\ref{thm:SimplicialHM}.
\end{enumerate}
\end{cor}

Holroyd, Spencer, and Talbot \cite{Holroyd/Spencer/Talbot:2005} showed
that if $G$ is the disjoint union of $d$ paths, cycles, and/or complete
graphs (including at least one isolated vertex), then $G$ is $r$-EKR
for all $r\leq d/2$. They did not address the strict EKR property. 
\begin{cor}
Let $G$ have $d$ connected components, and let $r<d/2$. (See e.g.
\cite[Lemma 6.11]{Jonsson:2008}.)
\begin{enumerate}
\item If $G$ has an isolated vertex, then it is strictly $r$-EKR.
\item If $G$ has at least $r$ isolated vertices, then it satisfies the
cross-intersecting condition of Theorems~\ref{thm:CrossIntClassic}
and \ref{thm:CrossIntShadow}.
\item If $G$ has at least $r+1$ isolated vertices, then it satisfies the
stability condition of Theorem~\ref{thm:SimplicialHM}.
\end{enumerate}
\end{cor}

\noindent Holroyd and Borg \cite{Borg/Holroyd:2009} later proved
stronger EKR results for a class of graphs similar to (slightly broader
than) that in \cite{Holroyd/Spencer/Talbot:2005}; their results include
a proof of Conjecture~\ref{conj:HolroydTalbotBorg} in this class
of graphs. 

A matroid is a certain combinatorial generalization of linear/affine
independence among sets of points in a vector space. A coloop in a
matroid is a point that is contained in every maximal independent
set.
\begin{cor}
Let $\Delta$ be the independence complex of a matroid. (See \cite{Provan/Billera:1980}.)
\begin{enumerate}
\item If $\Delta$ has a coloop, then it satisfies Conjecture~\ref{conj:HolroydTalbotBorg}.
\item If $\Delta$ has at least $r$ coloops, then it satisfies the cross-intersecting
condition of Theorems~\ref{thm:CrossIntClassic} and \ref{thm:CrossIntShadow}.
\item If $\Delta$ has at least $r+1$ coloops, then it satisfies the stability
condition of Theorem~\ref{thm:SimplicialHM}.
\end{enumerate}
\end{cor}

\subsection*{Organization}

This paper is organized as follows. In Section~\ref{sec:Background},
we overview the necessary background on simplicial complexes, depth,
and shifting. In Section~\ref{sec:Strict-EKR}, we prove our strict
EKR result, Theorem~\ref{thm:StrictEKR}. In Section~\ref{sec:HM},
we prove our stability result, Theorem~\ref{thm:SimplicialHM}. In
Section~\ref{sec:Cross-Intersecting}, we prove our cross-intersecting
results, Theorems~\ref{thm:CrossIntClassic} and \ref{thm:CrossIntShadow}.

\section{\label{sec:Background}Background}

\subsection{Simplicial complex basics}

An \emph{(abstract) simplicial complex} is a system $\Delta$ of sets
that is closed under inclusion. More precisely, elements of $\Delta$
are subsets of some universe set, and if $A\in\Delta$ and $B\subseteq A$,
then also $B\in\Delta$. For example, the \emph{$d$-simplex} is the
set of all subsets of a set with $d$ elements. The simplicial complexes
we consider will all be finite and nonempty.

The \emph{$f$-vector} or \emph{face vector} of $\Delta$ is the sequence
$f_{0},f_{1},\dots$, where $f_{0}=1$ is the number of $0$ element
sets in $\Delta$, $f_{1}$ is the number of $1$-element sets in
$\Delta$, and so forth. We take the nonstandard notation of $F_{k}(\Delta)$
for the set of $k$ element faces of $\Delta$, so that $f_{k}=\left|F_{k}\right|$.
Simplicial complexes are a basic object in algebraic topology \cite{Munkres:1984},
combinatorics \cite{Anderson:2002}, and in combinatorial algebra
and/or algebraic combinatorics \cite{Herzog/Hibi:2011,Stanley:1996,Villarreal:2015}.

The \emph{dimension} of a face $A$ is $\dim A:=\left|A\right|-1$.
The \emph{dimension} of a simplicial complex $\Delta$, denoted $\dim\Delta$,
is the maximum dimension of any face. A \emph{facet} is a maximal
(under set inclusion) face, so $\dim\Delta$ can also be described
as the maximum dimension of any facet. The \emph{$d$-dimensional
skeleton} of $\Delta$ is the subcomplex consisting of all faces of
dimension at most $d$.

There is a tension between dimension and cardinality in topological
combinatorics. While the tension is easily resolved by adding or subtracting
$1$, it can sometimes cause confusion. We will index the $f$-vector
by cardinality, and will say either \emph{$r$-face}, $r$ element
face,\emph{ }or $(r-1)$-dimensional face for a face containing $r$
vertices.

We say the \emph{link} of $A$ in $\Delta$ is the simplicial complex
\[
\link_{\Delta}A:=\left\{ B\in\Delta:A\cap B=\emptyset,A\cup B\in\Delta\right\} .
\]
An alternative description of $\link_{\Delta}A$ is that it is formed
by taking all sets in $\Delta$ that contain $A$, and removing the
subset $A$ from each of them. Thus, the link describes the faces
containing a vertex or other face.

The \emph{deletion} of a vertex $v$ is obtained by removing from
$\Delta$ all faces that contain $v$. Thus, 
\[
\del_{\Delta}v:=\{\tau\in\Delta:\tau\cap\{v\}=\emptyset\}.
\]
Note that deleting a set of vertices is a distinct operation from
removing a face from $\Delta$. It is easy to see that if $v$ is
a vertex and $A$ is a face not containing $v$, then deleting $v$
and taking the link of $A$ are commuting operations.

The \emph{independence complex }of a graph or simple hypergraph is
the family of all vertex subsets that contain no edges (the \emph{independent
sets}). Since a subset of an independent set is also independent,
the independence complex is indeed a simplicial complex.

\subsection{Depth and the Cohen-Macaulay property}

The invariant of depth of a simplicial complex is subtle, and yet
is surprisingly helpful in combinatorics.

We say that a simplicial complex $\Delta$ is \emph{Cohen-Macaulay}
over $\mathbb{F}$ (for $\ff$ a field or the integers) if 
\[
\tilde{H}_{i}(\link_{\Delta}A;\ff)=0\text{ for all }A\in\Delta\text{ and all }i<\dim\link_{\Delta}A.
\]
Here $\tilde{H}_{i}$ denotes the $i$th reduced simplicial homology
group from algebraic topology.

Note that the empty set is a face of every nondegenerate simplicial
complex, and that $\Delta=\link_{\Delta}\emptyset$. Thus, a simplicial
complex $\Delta$ is Cohen-Macaulay if and only if it satisfies the
following recursive formulation: either it is of dimension at most
$0$, or else both $\tilde{H}_{i}$ vanishes below the highest possibly-nonzero
dimension, and the link of every vertex is Cohen-Macaulay.

We typically suppress $\ff$ from the notation, as many examples of
importance are either Cohen-Macaulay over either every $\ff$ or over
no $\ff$.

The \emph{depth} of a simplicial complex $\Delta$, denoted $\depth\Delta$,
is the highest dimensional skeleton that is Cohen-Macaulay. Recall
here that the $d$-dimensional skeleton consists of all faces of dimension
at most $d$. Thus, a complex is Cohen-Macaulay if and only if $\depth\Delta=\dim\Delta$.
We remark that the Cohen-Macaulay property and depth arise from commutative
algebra through the ``face ring'' of $\Delta$, but that the commutative
algebraic depth and (Krull) dimension are both one larger in that
context. Indeed, this may be seen from one perspective as an example
of the tension between dimension and cardinality. See e.g. \cite[Appendix A]{Herzog/Hibi:2011}
for a development from the commutative algebra perspective.

It is sometimes useful to note that depth is a topological invariant
\cite{Munkres:1984b}. For example, any triangulation of a sphere
or ball is Cohen-Macaulay. Since $\tilde{H}_{0}(\Delta)$ is nonzero
if and only if $\Delta$ is disconnected, the depth of a connected
simplicial complex is at least $1$.

We observe from the recursive formulation of Cohen-Macaulay that:
\begin{cor}
\label{cor:DepthBoundLink}If $v$ is a vertex of the simplicial complex
$\Delta$, then $\depth(\link_{\Delta}v)\geq\depth\Delta-1$.
\end{cor}

A lower bound on the deletion also holds, though it is typically less
useful than the bound on the link. The following bound may be proved
with a standard long exact sequence argument; see e.g. \cite[Proposition 1.2.9]{Bruns/Herzog:1993}. 
\begin{lem}
\label{lem:DepthBoundDel}If $v$ is a vertex of the simplicial complex
$\Delta$, then $\depth(\del_{\Delta}v)\geq\depth\Delta-1$
\end{lem}

If we desire to specify a field for the Cohen-Macaulay property, we
may write $\depth_{\ff}\Delta$. An alternative formulation of depth
is 
\[
\depth_{\ff}\Delta=\max\{d:\tilde{H}_{i}(\link_{\Delta}A;\ff)=0\text{ for all }A\in\Delta\text{ and }i<d-\left|A\right|\}.
\]

It is a basic fact (and a straightforward exercise to prove) that
the facets of a Cohen-Macaulay simplicial complex all have the same
dimension. It follows immediately that the minimum facet dimension
of $\Delta$ is at least the depth. In the case where the two are
equal, we say that $\Delta$ has \emph{facet depth}.
\begin{cor}[of Theorem~\ref{thm:StrictEKR}]
 Near-cones with facet depth satisfy Conjecture~\ref{conj:HolroydTalbotBorg}.
\end{cor}

Additional background on the Cohen-Macaulay property may be found
in any survey or textbook on combinatorial commutative algebra or
topological combinatorics, for example \cite{Herzog/Hibi:2011,Kozlov:2008,Stanley:1996}.
Additional background on depth may be found in \cite{Herzog/Hibi:2011,Jonsson:2008,Woodroofe:2011a}.

\subsection{\label{subsec:VDcomplexes}Vertex-decomposable and sequentially Cohen-Macaulay}

A typical proof from the literature that a simplicial complex is Cohen-Macaulay
or has facet depth uses certain simple combinatorial decompositions.
It is often not necessary to work directly with the homological definition.

A simplicial complex is recursively defined to be \emph{vertex-decomposable}
if it is either
\begin{enumerate}
\item a simplex, or
\item has a vertex $v$ so that
\begin{enumerate}
\item if $v$ is in a face $A$, then there is some vertex $w\neq v$ so
that $A\setminus v\cup w$ is a face, and
\item both $\del_{\Delta}v$ and $\link_{\Delta}v$ are (recursively) vertex-decomposable.
\end{enumerate}
\end{enumerate}
A vertex satisfying the condition of (2) is called a \emph{shedding
vertex}.

Vertex-decomposability was introduced by Provan and Billera \cite{Provan/Billera:1980}
for the special case where all facets have the same dimension, and
extended by Björner and Wachs in \cite{Bjorner/Wachs:1997} to arbitrary
simplicial complexes.
\begin{example}
If $\Delta$ is the independence complex of a chordal graph, then
when $v$ is a neighbor of a simplicial vertex $w$, we may exchange
$v$ for $w$ in any independent set containing $v$. A simple inductive
argument yields that $\Delta$ is vertex-decomposable \cite{Dochtermann/Engstrom:2009,Woodroofe:2009a}.
A slightly more difficult argument \cite{Woodroofe:2009a} gives the
same for a graph where induced cycles have length $3$ or $5$.
\end{example}

\begin{example}
\label{exa:ShiftedIsVD}If $\Delta$ is shifted and not a simplex,
then the last vertex may be exchanged for an earlier vertex in any
face that it appears in. It follows easily that a shifted simplicial
complex is vertex-decomposable \cite[Section 11]{Bjorner/Wachs:1997}.
\end{example}

\begin{example}
If $\Delta$ is the independence complex of a matroid, then it follows
directly from the exchange condition that every vertex that is not
a coloop is a shedding vertex. Since links and deletions are also
independence complexes of matroids, it follows that $\Delta$ is vertex-decomposable
\cite[Section 3.2]{Provan/Billera:1980}.
\end{example}

The following is a basic result in topological combinatorics and combinatorial
commutative algebra.
\begin{prop}
\label{prop:VDhasFacetDepth}Any vertex-decomposable complex has facet
depth.
\end{prop}

More generally, vertex-decomposable complexes have the stronger property
of \emph{sequentially Cohen-Macaulay}, which means that for any $d$,
the subcomplex generated by all $d$-faces is Cohen-Macaulay. Since
facet depth only requires a single such subcomplex (at the size of
a minimal facet) to be Cohen-Macaulay, it is immediate that any sequentially
Cohen-Macaulay complex has facet depth. The property of \emph{shellability}
is also discussed extensively in the literature, and is weaker than
vertex-decomposability but stronger than sequentially Cohen-Macaulay.

Vertex-decomposability may also occasionally be used on a skeleton
of $\Delta$ to give a lower bound on depth.
\begin{example}
If $\Delta$ is the independence complex of a graph with $r$ connected
components, then if $A$ is an independent set with at most $r$ vertices,
and $w\in A$, then there is at least one connected component $C$
so that $A\setminus w$ contains no vertex of $C$. Thus, we see that
$A\setminus w\cup v$ is an independent set for any $v\in C$. It
follows that any non-isolated vertex is a shedding vertex, and a simple
induction yields that the $(r-1)$-dimensional skeleton is vertex-decomposable,
hence that $\depth\Delta\geq r-1$. This bound may be improved in
certain circumstances, as explained in \cite[Section 4]{Woodroofe:2011a}
or more generally in \cite[Section 6.3]{Jonsson:2008}.
\end{example}

Vertex-decomposability is often nicely covered in the combinatorial
commutative algebra literature \cite{Herzog/Hibi:2011,Villarreal:2015}.
Textbooks coming more from the topological and/or combinatorial point
of view such as \cite{Jonsson:2008,Kozlov:2008} unfortunately often
only treat the special case where all facets have the same dimension,
which is not always convenient.

\subsection{On shifted and shifting}

Typically, problems in combinatorial set theory are easier to solve
on shifted set systems. Here, a system $\mathcal{A}$ of subsets of
$\{1,2,\dots,n\}$ is \emph{shifted} if whenever $A\in\mathcal{A}$
is such that $j\in A$ and $i\notin A$ for $i<j$, then also $(A\setminus j)\cup i\in\mathcal{A}$.
Thus, a main proof strategy is to transform a given set system into
a shifted one (while preserving some structure from the original).
Indeed, the original proof of Erd\H{o}s, Ko, and Rado \cite{Erdos/Ko/Rado:1961}
followed this strategy.

A powerful method of transforming a set system into a shifted one
is the \emph{(exterior) algebraic shifting} shifting operation $\shiftext_{\ff}$
over a fixed infinite field $\ff$. We frequently suppress the field
$\ff$ from our notation. We describe the properties of this operation
first, and describe its construction below.
\begin{lem}
\label{lem:AlgebraicShiftingFacts}Let $\mathcal{A}$ be a set system,
and $\Gamma,\Delta$ be simplicial complexes.
\begin{enumerate}
\item \label{enu:AlgShift1stKalai}$\shiftext\mathcal{A}$ is a shifted
set system, with the same cardinality of that of $\mathcal{A}$.
\item If $\mathcal{A}$ is shifted, then $\shiftext\mathcal{A}=\mathcal{A}$.
\item \label{enu:AlgShiftLastKalai}If $\mathcal{A}$ consists of $k$ element
sets, then $\bdry\shiftext\mathcal{A}\subseteq\shiftext\bdry\mathcal{A}$.
In particular, the shift of a simplicial complex is a simplicial complex.
\item \label{enu:AlgShiftIntersect}If $\mathcal{A}$ is intersecting, then
so is $\shiftext\mathcal{A}$. Similarly, if $\mathcal{A}$ is cross-intersecting
with another set system $\mathcal{B}$, then $\shiftext\mathcal{A}$
and $\shiftext\mathcal{B}$ are also cross-intersecting.
\item \label{enu:AlgShiftSimplicial}If $\Gamma\subseteq\Delta$, then $\shiftext\Gamma\subseteq\shiftext\Delta$.
\item \label{enu:AlgShiftingDepthDim}$\depth_{\ff}\Delta=\depth_{\ff}\shiftext_{\ff}\Delta$.
Thus, since shifted complexes have facet depth, the dimension of a
minimum facet of $\shiftext_{\ff}\Delta$ is exactly $\depth_{\ff}\Delta$.
\item \label{enu:AlgShiftingHomol}The homology of $\Delta$ is preserved
by shifting. That is, 
\[
\tilde{H}_{i}(\Delta;\ff)\cong\tilde{H}_{i}(\shiftext_{\ff}\Delta;\ff)\text{ for all }i.
\]
\item \label{enu:AlgShiftingNC}If $\Delta$ is a near-cone with apex vertex
$a$, then the $f$-vectors of $\link_{\Delta}a$ and $\del_{\Delta}a$
coincide with those of the apex vertex of $\shiftext\Delta$.
\end{enumerate}
\end{lem}

Here, parts (\ref{enu:AlgShift1stKalai})-(\ref{enu:AlgShiftLastKalai})
and (\ref{enu:AlgShiftSimplicial}) may be found (with further references)
in \cite[Section 2.5]{Kalai:2002}, and (\ref{enu:AlgShiftIntersect})
is in \cite[Section 6.4]{Kalai:2002}. Preservation of depth (\ref{enu:AlgShiftingDepthDim})
follows from \cite[Section 4.2]{Kalai:2002}, or a clear statement
is in \cite[Corollary 4.5]{Duval:1996}. Preservation of homology
(\ref{enu:AlgShiftingHomol}) is stated in \cite[Section 3.3]{Kalai:2002},
or an algebraic proof is given in \cite[Section 11.4]{Herzog/Hibi:2011}.
Part (\ref{enu:AlgShiftingNC}) is considered in \cite[Section 2.5]{Kalai:2002}
for the case of a cone (where every facet contains the apex vertex),
and is extended to near-cones in \cite{Nevo:2005}. See also the discussion
in \cite[Section 2.2]{Woodroofe:2011a}; note that since $\shiftext\Delta$
is shifted, it is in particular a near-cone.

Thus, the algebraic shift operation replaces a given simplicial complex
or other set system with a shifted one in a single step. We refer
the reader to \cite{Herzog/Hibi:2011,Kalai:2002} for general background
on algebraic shifting; Duval made particular use of the fact that
algebraic shifting preserves depth in \cite{Duval:1996}.
\begin{example}
Consider the simplicial complex $\Delta$ consisting of all subsets
of the sets $\left\{ 1,2,3\right\} $ and $\left\{ 1,4,5\right\} $.
Geometrically, this may be described as two triangles sharing a vertex;
the $f$-vector is $1,5,6,2$, and it is possible to show that the
depth is $1$. The only shifted family containing two $3$ element
sets consists of $\left\{ 1,2,3\right\} $ and $\left\{ 1,2,4\right\} $.
Since the shadows of this family have only $5$ edges and $4$ vertices,
we see that $\shiftext\Delta$ has also an extra vertex $5$, and
an extra edge, necessarily $\{1,5\}$. In particular, a shifted simplicial
complex with the same $f$-vector as $\Delta$ must have a facet containing
fewer than 3 vertices.
\end{example}

An immediate corollary of Lemma~\ref{lem:AlgebraicShiftingFacts}~(\ref{enu:AlgShiftingHomol})
is as follows. We say that a family $\mathcal{A}$ of $k$ element
sets \emph{spans a simplex boundary} if it contains all $k$ element
subsets of some $k+1$ element set. That is, the family $\mathcal{A}$
spans a simplex boundary if it contains all of the facets of some
$k$-dimensional simplex.
\begin{cor}
\label{cor:ShiftedNontrivial}If $\mathcal{A}$ spans a simplex boundary,
then $\shift^{\wedge}\mathcal{A}$ has the same property.
\end{cor}

In particular, if $\mathcal{A}$ spans a simplex boundary, then $\shiftext\mathcal{A}$
has no common intersection.
\begin{rem}
\label{rem:AlgShiftConstruction}For completeness, we sketch one construction
of exterior algebraic shifting. This construction uses more machinery
than some other descriptions, but the setting may make it more plausible
that the operation is well-behaved. The details will not be important
in the rest of the paper, and can be skipped on a first reading.

We model our subsets of $\{1,2,,\dots,n\}$ as monomials in an exterior
algebra $\extalg\ff^{n}$, and our set system as a subspace $L$ of
the exterior algebra. This is convenient for several reasons, one
being that multiplication allows us to detect whether two sets intersect
nontrivially or not. We then replace the standard basis $e_{1},\dots e_{n}$
with a generic basis (satisfying no nontrivial algebraic equations)
$g_{1},\dots,g_{n}$. In this basis, we act by the diagonal matrix
$D$ with $t^{1},t^{4},\dots,t^{2^{2^{n-1}}}$ in its entries, extend
the action to forms in $\extalg\ff^{n}$ and subspaces thereof, and
take a limit as $t\to0$. We projectivize, so that e.g. $g_{1}+tg_{2}\equiv\frac{1}{t}g_{1}+g_{2}$.
Then in this limit, the term in an element that survives is the one
that is assigned the lowest power of $t$ by $D$. That is, the term
that survives is exactly the lexicographically first monomial. (This
term is unique by choice of powers of $t$ in the diagonal.) By the
lexicographic condition, the result corresponds to a shifted set system,
which is the exterior algebraic shift.

Since we are moving a representation of our system along a curve in
a space, certain ``Zariski closed'' properties are preserved. (In
algebraic geometry, this corresponds to the ``flat degeneration''.)
\end{rem}

Algebraic shifting is powerful, but relatively opaque. As a result,
we combine it with a second shifting operation. The \emph{combinatorial
shift} from $w$ to $v$ is the operation $\shift_{v\leftarrow w}$
on a set system $\mathcal{A}$, defined as 
\begin{align*}
\shift_{v\leftarrow w}\mathcal{A}= & \left\{ A\in\mathcal{A}:w\notin A\text{ or }v\in A\text{ or }A\setminus w\cup v\in\mathcal{A}\right\} \\
 & \cup\left\{ A\setminus w\cup v:w\in A\text{ and }v\notin A\right\} .
\end{align*}
Thus, $\shift_{v\leftarrow w}$ replaces $w$ with $v$ in a set $A$
containing $w$ but not $v$, so long as $A\setminus w\cup v$ is
not already in the system. It is a basic result of the field \cite{Frankl:1987}
that a system of subsets of $[n]$ stabilizes under iterative application
of shifting operations $\shift_{i\leftarrow j}$ with $i<j$, and
it is immediate from definitions that a system that is invariant under
all such $\shift_{i\leftarrow j}$ is a shifted system. Combinatorial
shifting is surveyed in \cite{Frankl:1987}. It is compared with algebraic
shifting in \cite{Murai/Hibi:2009}, as is also described in \cite[Chapter 11]{Herzog/Hibi:2011}.

The combinatorial shift of a complex is computable, and is in certain
circumstances a little easier to control than algebraic shifting.
On the other hand, we will need to apply many combinatorial shifts
to reduce to a shifted system. Understanding the interactions between
a large number of $\shift_{i\leftarrow j}$ operations is not trivial,
and the algebraic shift can be more convenient to work with than the
aggregate of combinatorial shifts.

As observed by Murai and Hibi \cite{Murai/Hibi:2009}, combinatorial
shifting may also in some circumstances be less well-behaved than
algebraic shifting. We give a more concrete example of this phenomenon.
\begin{example}
Let $\Delta$ be the $1$-dimensional simplicial complex with facets
$\{1,2\},\{2,3\},$ $\{3,4\}$. Since $\Delta$ is the independence
complex of a tree, it is Cohen-Macaulay (this is also easy to check
directly). After performing $\shift_{1\leftarrow4}$, we are left
with facets $\{1,2\},\{1,3\},\{2,3\},\{4\}$. We see that combinatorial
shifting may not preserve minimal facet size, even in a complex with
facet depth. On the other hand, it follows by examination of Lemma~\ref{lem:AlgebraicShiftingFacts}
that $\shiftext\Delta$ has facets $\{1,2\},\{1,3\},\{1,4\}$.
\end{example}

\begin{rem}
Work of Murai and Hibi \cite{Murai/Hibi:2009}, of Knutson \cite{Knutson:2014UNP}
and of the authors with Gandini \cite{Bulavka/Gandini/Woodroofe:2024UNP}
describes combinatorial shifting in a similar algebraic geometry setting
to that of algebraic shifting. A recent preprint of Della Vecchia,
Joswig, and Lenzen \cite{DellaVecchia/Joswig/Lenzen:2024UNP} shows
that a suitable (non-generic) basis in the situation of Remark~\ref{rem:AlgShiftConstruction}
may also be used to realize combinatorial shifting.
\end{rem}

\subsection{Near-cones and partially shifted complexes}

Our main theorems all require simplicial complexes that are shifted
with respect to one or more vertices.

The starting point for this is a near-cone with apex vertex $v_{1}$,
which (as previously discussed) is a complex that is invariant under
all operations $\shift_{v_{1}\leftarrow w}$. The terminology comes
about because of a particular special case: a \emph{cone} is a simplicial
complex where every facet contains an apex vertex $v_{1}$. We see
that if $\Delta$ is a near-cone, then $\Delta$ consists of a cone
$\Delta'$, together with additional facets whose boundaries lie in
$\Delta'$. We record some easy facts about near-cones:
\begin{lem}
\label{lem:NearConeFacts} If $\Delta$ is a near-cone with apex vertex
$a$, then:
\begin{enumerate}
\item For any $r$, the quantity $f_{r}(\link_{\Delta}v)$ is maximized
at $v=a$.
\item If $\mathcal{A}$ is a system of $r$ element faces of $\Delta$,
then after performing the $\shift_{a\leftarrow w}$ operation on $\mathcal{A}$
for each $w\neq a$, we obtain a system $\mathcal{A}'$ that is shifted
with respect to $a$.
\end{enumerate}
\end{lem}

More generally, a simplicial complex $\Delta$ is \emph{shifted with
respect to $\{v_{1}<\cdots<v_{t}\}$} if $\Delta$ is invariant under
all $\shift_{v_{i}\leftarrow w}$ where $w\notin\{v_{1},\dots,v_{i}\}$.
This property may be expressed recursively in terms of the near-cone
property, as follows. The complex $\Delta$ is a \emph{$t$-fold near-cone}
with respect to $\left\{ v_{1}<\cdots<v_{t}\right\} $ if either $t=0$,
or if $\Delta$ is a near-cone with apex vertex $v_{1}$ such that
both $\del_{\Delta}v_{1}$ and $\link_{\Delta}v_{1}$ are $(t-1)$-fold
near-cones with respect to $\{v_{2}<\cdots<v_{t}\}$.
\begin{lem}
The simplicial complex $\Delta$ is shifted with respect to $\{v_{1}<\cdots<v_{t}\}$
if and only if it is a $t$-fold near-cone with respect to $\{v_{1}<\cdots<v_{t}\}$.
\end{lem}

\begin{proof}
It is immediate from definition that of $\shift_{v\leftarrow w}$
that $\Delta$ is shifted with respect to $\{v_{1}<v_{2}<\cdots<v_{t}\}\neq\emptyset$
if and only if $\Delta$ is a near-cone with respect to $v_{1}$,
and both $\del_{\Delta}v_{1}$ and $\link_{\Delta}v_{1}$ are shifted
with respect to $\{v_{2}<\cdots<v_{t}\}$. The lemma now follows by
a straightforward induction.
\end{proof}
We remark that Seyed Fakhari \cite{Fakhari:2017} defined a simplicial
complex $\Delta$ to be a ``$t$-near-cone'' if $t=0$, or if $\Delta$
is a near-cone with apex vertex $v_{1}$ such that $\del_{\Delta}v_{1}$
is a $(t-1)$-near-cone. Every $t$-fold near-cone in our sense is
a $t$-near-cone in Seyed Fakhari's sense, but the converse is not
true.

The most convenient way to construct a simplicial complex that is
shifted with respect to $\{v_{1}<\cdots<v_{t}\}$ is to take a $t$-fold
cone; that is, a complex where every facet contains $\{v_{1},\dots,v_{t}\}$.
Of course, in this case, the order of $v_{1},\dots,v_{t}$ does not
matter. An independence complex of a graph $G$ is a $t$-fold cone
if and only if $G$ has $t$ isolated vertices. However, there are
independence complexes of graphs that are shifted with respect to
an initial set of vertices, but that are not cones. Indeed, Klivans
showed in \cite{Klivans:2007} that the independence complex of a
graph $G$ is shifted if and only if $G$ is a threshold graph.

The following generalizes the fact that combinatorial shifting stabilizes
when taken over all pairs of elements.
\begin{lem}
\label{lem:CombShiftNearcone} Let $\Delta$ be shifted with respect
to $\{v_{1}<v_{2}<\cdots<v_{t}\}$. If $\mathcal{A}$ is a system
of faces of $\Delta$, then after performing the $\shift_{v_{i}\leftarrow w}$
operation on $\mathcal{A}$ for each $w\notin\left\{ v_{1},\dots,v_{i}\right\} $,
we obtain a system of faces $\mathcal{A}'$ that is shifted with respect
to $\{v_{1}<v_{2}<\cdots<v_{t}\}$.
\end{lem}

\subsection{Hilton-Milner}

We will also need a cross-intersecting result of Hilton and Milner,
which will be more general than the cross-intersecting theorem discussed
in the introduction (but still a special case of their most general
result). Recall that a \emph{Sperner family} (also called a \emph{clutter}
or \emph{simple hypergraph}) is a system of subsets with no proper
inclusions.
\begin{thm}[{Hilton and Milner \cite[$p=1$ case of Theorem~2]{Hilton/Milner:1967}}]
\label{thm:HiltonMilner}Let $\mathcal{B}$ and $\mathcal{C}$ be
nonempty Sperner families of subsets of $\{1,\dots,n\}$. If for each
$B\in\mathcal{B}$ and $C\in\mathcal{C}$ we have $\left|B\right|,\left|C\right|\leq r$
and $B\cap C\neq\emptyset$, then 
\[
\left|\mathcal{B}\right|+\left|\mathcal{C}\right|\leq{n \choose r}-{n-r \choose r}+1.
\]
\end{thm}

Thus, the uniform case of Theorem~\ref{thm:HiltonMilner} is the
special case of Theorem~\ref{thm:CrossIntClassic} on a simplex.

Indeed, Theorem~\ref{thm:CrossIntClassic} also generalizes to Sperner
families of faces. Hibi proved the following lemma in \cite[Section 2]{Hibi:1989}
for hereditary multiset systems, but we state for simplicial complexes.
\begin{lem}[Hibi \cite{Hibi:1989}]
\label{lem:HibiContainInject} Let $\Delta$ be a simplicial complex
whose smallest facet has $d$ vertices. If $s\leq r\leq d-s$, then
there is an injection $\psi:F_{s}(\Delta)\to F_{r}(\Delta)$ such
that $A\subseteq\psi(A)$.
\end{lem}

Now if a Sperner family $\mathcal{A}$ consists of faces of $\Delta$
having at most $r$ vertices, then we may repeatedly apply Lemma~\ref{lem:HibiContainInject}
to augment the sets that are smaller than $r$. See also \cite[Section 3]{Borg:2009}.
We obtain the following generalization of Theorem~\ref{thm:CrossIntClassic}.
\begin{thm}
\label{thm:CrossIntClassicSperner}Let $\Delta$ be a simplicial complex
of depth $d-1$, and let $2\leq r\leq d/2$. Let $\mathcal{A},\mathcal{B}$
be nonempty cross-intersecting Sperner families consisting of faces
of $\Delta$, each with at most $r$ vertices.

If $\Delta$ is shifted with respect to $\{v_{1}<\cdots<v_{r}\}$,
then 
\[
\left|\mathcal{A}\right|+\left|\mathcal{B}\right|\leq f_{r}(\Delta)-\gamma+1,
\]
where $\gamma$ is the number of $r$ element faces containing no
element of $\{v_{1},\dots,v_{r}\}$.
\end{thm}

\section{\label{sec:Strict-EKR}Strict EKR}

In this section, we prove Theorem~\ref{thm:StrictEKR}. We follow
the general strategy discussed in Section~\ref{subsec:MainStrat}
of the introduction.

We first fix some notation. For some $r<d/2$, let $\mathcal{A}$
be an intersecting family of $r$-faces with no common intersection.
We wish to show that $\left|\mathcal{A}\right|$ is strictly smaller
than $f_{r-1}(\link_{\Delta}a)$, where $\Delta$ is shifted with
respect to $a$. The result is trivial for $r=1$, so we may assume
$r\geq2$.\vspace{-0.06cm}

\subsection{Reduction}

Since $\Delta$ is shifted with respect to $a$, we may perform combinatorial
shifting operations $\shift_{a\leftarrow w}$ on $\mathcal{A}$. We
do so as long as the operations do not result in a common intersection.

If no such operation results in a common intersection, then by Lemma~\ref{lem:NearConeFacts},
the system stabilizes in a system $\mathcal{A}'$ that is shifted
with respect to $a$. Since the stabilized system has no common intersection,
we have a set $A\in\mathcal{A}'$ with $a\notin A$. By definition
of $\shift_{a\leftarrow w}$, we see that $\mathcal{A}$ spans a simplex
boundary on $a\cup A$.

Corollary~\ref{cor:ShiftedNontrivial} gives that $\shiftext\mathcal{A}$
has no common intersection. Now Lemma~\ref{lem:AlgebraicShiftingFacts}
reduces to the case where $\Delta$ is shifted.

Otherwise, we have an operation $\shift_{a\leftarrow v}$ that yields
a system with common intersection.\vspace{-0.06cm}

\subsection{Simpler cases}

We consider separately the two cases that we reduced to above.

The first simple case is where $\Delta$ and $\mathcal{A}$ are shifted.
In this case, the theorem is a result of Borg \cite[Theorem 2.7]{Borg:2009}.
Alternatively, since shifted complexes have facet depth by Proposition~\ref{prop:VDhasFacetDepth},
this is a special case of Theorem~\ref{thm:SimplicialHM}. It is
of note that a shifted system $\mathcal{A}$ with no common intersection
is covered by the last part of the proof of Theorem~\ref{thm:SimplicialHM}.

The second simpler case is where $\mathcal{A}$ has no common intersection,
but where $\shift_{a\leftarrow v}\mathcal{A}$ does have a common
intersection, necessarily at $a$. This case is parallel to Case~2
of the proof of Theorem~\ref{thm:SimplicialHM}.

Since $\shift_{a\leftarrow v}\mathcal{A}$ has a common intersection
at $a$, every face of $\mathcal{A}$ must contain either $v$, or
$a$, or both. Moreover, if $a\notin A$ for a set $A\in\mathcal{A}$,
then $A\setminus v\cup a$ is not in $\mathcal{A}$, as otherwise
$A$ would be preserved in $\shift_{a\leftarrow v}\mathcal{A}$, contradicting
the common intersection assumption. We obtain that the following map
is an injection:\vspace{-0.18cm}

\begin{align*}
\varphi:\mathcal{A} & \to F_{r-1}(\link_{\Delta}a)\\
A & \mapsto\begin{cases}
A\setminus a & \text{if }a\in A\\
A\setminus v & \text{otherwise}.
\end{cases}
\end{align*}
To complete the proof in this case, it is enough to show that $\varphi$
is not surjective. Since without loss of generality, we may have all
faces containing $\left\{ a,v\right\} $ in $\mathcal{A}$, we focus
on the faces of the image that do not contain $v$.

Let $T$ be a facet of $\Delta$ that contains $a$, and that contains
at least one $C_{0}$ in $\mathcal{A}$ with $a\notin C_{0}$; thus,
$v\in C_{0}\subseteq T$. Such a $T$ clearly exists: indeed, since
$\Delta$ is a near-cone whose smallest facets have more than $r$
vertices, we may find a facet containing $a\cup A$ for any $A\in\mathcal{A}$.
We notice that all subsets of $T\setminus\{a,v\}$ are faces of $\link_{\Delta}a$
that avoid $v$.

We now consider the following Sperner families of subsets of $T\setminus\{a,v\}$.
\begin{align*}
\mathcal{B}_{T} & \text{ consisting of the maximal sets in \ensuremath{\left\{  \varphi(B)\cap T\,:\,v\notin B,B\in\mathcal{A}\right\} } }\\
\mathcal{C}_{T} & =\left\{ \varphi(C)\,:\,a\notin C,C\in\mathcal{A},\varphi(C)\subseteq T\right\} .
\end{align*}
These Sperner families are cross-interesting, since if $B,C$ in $\mathcal{A}$
are such that $v\notin B$ and $a\notin C$, then $B$ and $C$ intersect
at a vertex other than $a$ or $v$, hence (since $C\setminus v\subseteq T$)
at a vertex of $T$. It follows from the same argument that $\mathcal{B}_{T}$
is nonempty, and we note that $\mathcal{C}_{T}$ is nonempty by choice
of $T$.

We now apply the (classical) cross-intersecting bound of Theorem~\ref{thm:HiltonMilner}:

\[
\left|\mathcal{B}_{T}\right|+\left|\mathcal{C}_{T}\right|\leq{\left|T\right|-2 \choose r-1}-{\left|T\right|-r-1 \choose r-1}+1<{\left|T\right|-2 \choose r-1}.
\]
Here, the strict inequality holds since by the depth condition $\left|T\right|\geq d$,
so as $d\geq2r+1$, also ${\left|T\right|-r-1 \choose r-1}$ is greater
than ${r \choose r-1}$.

Finally, the $(r-1)$-sets in $\mathcal{B}_{T}\cup\mathcal{C}_{T}$
are exactly the subsets of $T\setminus\{a,v\}$ that lie in the image
of $\varphi$. Since (by the above count) at least one such subset
is missing, the injection $\varphi:\mathcal{A}\to F_{r-1}(\link_{\Delta}a)$
is not surjective, completing the proof.
\begin{rem}
An alternative approach to the case where $\shift_{a\leftarrow v}$
yields a common intersection is to consider the cross-intersecting
families of $(r-1)$-element faces $\mathcal{B}=\left\{ B\setminus a:v\notin B,B\in\mathcal{A}\right\} $
and $\mathcal{C}=\left\{ C\setminus v:a\notin C,C\in\mathcal{A}\right\} $,
both contained in $\link_{\del_{\Delta}v}a=\del_{\link_{\Delta}a}v$.
Corollary~\ref{cor:DepthBoundLink} and Lemma~\ref{lem:DepthBoundDel}
tell us that passing to the deletion/link reduces the depth by at
most $2$, while we have reduced the size of our sets by $1$. Algebraically
shifting and applying Theorem~\ref{thm:CrossIntClassic} (noting
that $\gamma>1$) now gives the desired.

\end{rem}

\section{\label{sec:HM}Stability}

In this section, we prove our Hilton-Milner type theorem, Theorem~\ref{thm:SimplicialHM}.
Our proof will reduce to Theorem~\ref{thm:CrossIntShadow}, which
we prove in the following section.

The theorem is trivial for $r=1$, so we assume without loss of generality
that $r\geq2$. We follow our main strategy: first shifting combinatorially
to reduce to the case where $\mathcal{A}$ spans a simplex boundary,
and then shifting algebraically. Indeed, we perform shifts $\shift_{v_{1}\leftarrow w}$
over all $w\neq v_{1}$. Recall that if $v$ is a vertex of $\Delta$,
then $\depth(\link_{\Delta}v)\geq\depth\Delta-1$. There are two cases:

\medskip{}

\emph{Case 1}: No shift operation $\shift_{v_{1}\leftarrow w}$ yields
a system with common intersection. Then by Lemma~\ref{lem:CombShiftNearcone},
we end in a system $\mathcal{A}'$ that has no common intersection,
and that is shifted with respect to $v_{1}$. Now if $A$ is any face
in $\mathcal{A}'$ that does not contain $v_{1}$, then $\mathcal{A}'$
spans a simplex boundary on $A\cup v_{1}$.

\medskip{}

\emph{Case 2}: The shift operation $\shift_{v_{1}\leftarrow w}$ yields
a system with a common intersection, and we stop just before that
operation. Thus, we have a system where every set contains either
$v_{1}$, $w$, or both. We may assume without loss of generality
that the resulting system has all $r$ element faces with both $v_{1}$
and $w$; otherwise, add these faces to the system.

Let $I=\left\{ v_{2},\cdots,v_{r+1}\right\} \setminus w$, and let
$I_{0}$ be the first $r-1$ vertices of $I$. We may now perform
additional shifts $\shift_{v_{i}\leftarrow x}$, where $v_{i}$ is
in the set $I$, and $x\notin\left\{ v_{1},\dots,v_{i},w\right\} $,
until the system stabilizes in $\mathcal{A}''$. Since we have at
least one face containing $v_{1}$ and not $w$ (and vice-versa),
we have both $I_{0}\cup v_{1}$ and $I_{0}\cup w$ in $\mathcal{A}''$
by the shiftedness property. Since we also have all faces with both
$v_{1}$ and $w$, the system $\mathcal{A}''$ spans a simplex boundary
on $I_{0}\cup v_{1}\cup w$. \medskip{}

Once we have a system that spans a simplex boundary, we may algebraically
shift. By Corollary~\ref{cor:ShiftedNontrivial}, we obtain a shifted
system $\mathcal{A}''$ with no common intersection. We notice that
since $\Delta$ is shifted with respect to $\left\{ v_{1},\dots,v_{r+1}\right\} $,
by iterative application of Lemma~\ref{lem:AlgebraicShiftingFacts}~(\ref{enu:AlgShiftingNC}),
the face numbers of links/deletions of $\left\{ v_{1},\dots,v_{r+1}\right\} $
are preserved under algebraic shifting. In particular, the quantity
$f_{r-1}-\beta+1$ in the statement of the theorem is preserved under
algebraic shifting.

We finish by applying Theorem~\ref{thm:CrossIntShadow} to the families
\begin{align*}
\mathcal{C} & :=\left\{ A\setminus v_{1}\,:\,A\in\mathcal{A}'',\text{ with }v_{1}\in A\right\} \text{ and}\\
\mathcal{D} & :=\left\{ A\text{\ensuremath{\qquad}}:\,A\in\mathcal{A}'',\text{ with }v_{1}\notin A\right\} 
\end{align*}
of faces in $\link_{\Delta}v_{1}$. Note that by shiftedness, we have
both that $\bdry\mathcal{D}\subseteq\mathcal{C}$, and also that $\mathcal{D}$
can be viewed as sitting inside $F_{r}(\link_{\Delta}v_{1})$. The
result follows.

\subsection{Counterexamples of Borg}

Borg in \cite{Borg:2013} shows that $t$-fold cones over certain
disjoint unions of simplices do not obey the condition of Theorem~\ref{thm:SimplicialHM}
for $r\geq t\geq3$. We observe that, since a disconnected complex
has zero depth, these examples have a relatively low depth of $t$.
We do not know to what extent the shiftedness condition of Theorem~\ref{thm:SimplicialHM}
can be weakened.

\section{\label{sec:Cross-Intersecting}Proofs of cross-intersecting theorems}

As in the proof of Theorem~\ref{thm:SimplicialHM}, if $\Delta$
is shifted with respect to $\left\{ v_{1},\dots,v_{r}\right\} $,
then by iterative application of Lemma~\ref{lem:AlgebraicShiftingFacts}~(\ref{enu:AlgShiftingNC}),
the face numbers of links/deletions of $\left\{ v_{1},\dots,v_{r}\right\} $
are preserved under algebraic shifting. Thus, the quantities $f_{r}-\gamma+1$
in both theorems are preserved under algebraic shifting.

Since we do not need to preserve the delicate ``no common intersection''
property, we will not need combinatorial shifting, and use the partial
shiftedness property only to preserve the quantities $\gamma$.

Given a set system $\mathcal{A}$ on ordered vertex set $\left\{ v_{1}<\cdots<v_{n}\right\} $,
write $\mathcal{A}(\neg n)$ for the subfamily of all sets (faces)
in $\mathcal{A}$ that do not contain $v_{n}$, and $\mathcal{A}(n)$
for the family of all sets (faces) $A\setminus v_{n}$ over all $A\in\mathcal{A}$
that contain $v_{n}$. The following lemma is well-known:
\begin{lem}
\label{lem:CrossIntersectingAtVn}Let $\mathcal{A},\mathcal{B}$ be
shifted, cross-intersecting systems of subsets of $\left\{ v_{1}<\cdots<v_{n}\right\} $.
If each set in $\mathcal{A}$ and $\mathcal{B}$ contains at most
$n/2$ vertices, then $\mathcal{A}(\neg n)$ and $\mathcal{B}(\neg n)$
are also cross-intersecting.
\end{lem}

\begin{proof}
If $A\cap B=\left\{ v_{n}\right\} $, then $\left(A\setminus v_{n}\cup v_{i}\right)\cap B=\emptyset$
for any $v_{i}\notin A\cup B$.
\end{proof}

\subsection{Proof of Theorem~\ref{thm:CrossIntShadow}}

By algebraic shifting, it is enough to prove for a shifted complex
on ordered vertex set $\left\{ v_{1}<\cdots<v_{n}\right\} $, and
for shifted sets $\mathcal{A},\mathcal{B}$.

We induct on $n$, considering $\link_{\Delta}v_{n}$ (containing
faces from $\mathcal{A}(n)$, $\mathcal{B}(n)$) and $\del_{\Delta}v_{n}$
(containing faces from $\mathcal{A}(\neg n)$, $\mathcal{B}(\neg n)$).
The base cases are $r=1$ and where $\Delta$ is a simplex. Notice
that by shiftedness, $\mathcal{A}(\neg n)$ contains $\left\{ v_{1},\dots,v_{r-1}\right\} $
and $\mathcal{B}(\neg n)$ contains $\left\{ v_{1},\dots,v_{r}\right\} $.

Since in a shifted complex $\depth\left(\del_{\Delta}v_{n}\right)\geq\depth\Delta$,
we have by induction that 
\[
\left|\mathcal{A}(\neg n)\right|+\left|\mathcal{B}(\neg n)\right|\leq f_{r}(\del_{\Delta}v_{n})-\gamma_{r}(\neg n)+1,
\]
where $\gamma_{r}(\neg n)$ is the number of $r-1$ element faces
in $\del_{\Delta}v_{n}$ containing no element of $\left\{ v_{1},\dots,v_{r}\right\} $.

If $\mathcal{A}(n)$ is empty, then the shadow condition yields that
$\mathcal{B}(n)$ is also empty. If $\mathcal{B}(n)$ is empty, then
since $\left\{ v_{1},\dots,v_{r}\right\} $ is in $\mathcal{B}(\neg n)$,
every element of $\mathcal{A}(\neg n)$ must intersect $\left\{ v_{1},\dots,v_{r}\right\} $,
and the result follows immediately.

Finally, if both $\mathcal{A}(n)$ and $\mathcal{B}(n)$ are nonempty,
then Lemma~\ref{lem:CrossIntersectingAtVn} yields that they are
cross-intersecting. Since the minimal facet size of $\link_{\Delta}v_{n}$
is at most one smaller than that of $\Delta$, induction gives us
that 
\[
\left|\mathcal{A}(n)\right|+\left|\mathcal{B}(n)\right|\leq f_{r-1}(\link_{\Delta}v_{n})-\gamma_{r-1}(n)+1,
\]
 where $\gamma_{r-1}(n)$ is the number of $r-2$ element faces in
$\link_{\Delta}v_{n}$ containing no element of $\left\{ v_{1},\dots,v_{r-1}\right\} $.
As $\gamma_{r-1}(n)$ is at least one larger than the number $\gamma_{r}(n)$
of $r-2$ element faces in $\link_{\Delta}v_{n}$ containing no element
of $\left\{ v_{1},\dots,v_{r}\right\} $, and as $\gamma_{r}(n)+\gamma_{r}(\neg n)=\gamma$,
the result follows.

\subsection{Proof of Theorem~\ref{thm:CrossIntClassic}}

The proof is similar to that of Theorem~\ref{thm:CrossIntShadow},
and we highlight only a few of the differences.

By algebraic shifting, it is enough to prove for a shifted simplicial
complex and shifted families $\mathcal{A},\mathcal{B}$. The cases
where $\mathcal{A}(n)$ and $\mathcal{B}(n)$ are empty are completely
symmetric, and are similar to the case in the proof of Theorem~\ref{thm:CrossIntShadow}
where $\mathcal{B}(n)$ is empty.

The proof is otherwise entirely similar to that of Theorem~\ref{thm:CrossIntShadow},
and we omit the details.
\begin{rem}
Alternatively, one can reduce to shifted as above, and then appeal
to \cite[Theorem 1.2]{Borg:2010}.
\end{rem}

\bibliographystyle{3_Users_russw_Documents_Research_mypapers_Stric___theorems_for_simplicial_complexes_hamsplain}
\bibliography{2_Users_russw_Documents_Research_mypapers_Stric___do_theorems_for_simplicial_complexes_Master}

\end{document}